\documentclass[12pt]{amsart} 
\usepackage{amssymb} 
\usepackage{amsmath}
\usepackage{eucal} 
\usepackage{fullpage} % for long equations to fit on one line
\usepackage{eulervm}
\usepackage{url}
\theoremstyle{plain}% default
\newtheorem{thm}{Theorem}[section] 
\newtheorem{lem}[thm]{Lemma}
\newtheorem{prop}[thm]{Proposition} 
\newtheorem{cor}[thm]{Corollary}
\newtheorem{conj}[thm]{Conjecture}

\theoremstyle{definition}
\newtheorem{defn}{Definition}[section]

\newcommand{\bpartial}{\boldsymbol{\partial}}

\author{Mirk\'o Visontai} 
\address{Department of Mathematics, Royal Institute of Technology, SE-100 44 Stockholm, Sweden}
\email{visontai@kth.se}
\author{Nathan Williams}
\address{School of Mathematics, University of Minnesota, Minneapolis, MN 55455}
\email{will3089@math.umn.edu}

\title{Stable multivariate $W$-Eulerian polynomials} 

\begin{document}
\begin{abstract}
We prove a multivariate strengthening of Brenti's result that every root 
of the Eulerian polynomial of type $B$ is real. Our proof combines a 
refinement of the descent statistic for signed permutations with the 
notion of real stability---a generalization of real-rootedness to 
polynomials in multiple variables. The key is that our refined multivariate
Eulerian polynomials satisfy a recurrence given by a 
stability-preserving linear operator.

Our results extend naturally to colored permutations, and we also give
stable generalizations of recent real-rootedness results due to Dilks, Petersen, and 
Stembridge on affine Eulerian polynomials of types $A$ and $C$. Finally, 
although we are not able to settle Brenti's real-rootedness conjecture 
for Eulerian polynomials of type $D$, nor prove a companion conjecture 
of Dilks, Petersen, and Stembridge for affine Eulerian polynomials of types 
$B$ and $D$, we indicate some methods of attack and pose some related open problems.
\end{abstract}

\maketitle

\section{Introduction}
In this paper, we study the real-rootedness property of Eulerian polynomials 
for Coxeter groups from a combinatorial perspective. There is a well-known 
combinatorial interpretation of the Eulerian polynomial $A_n(x)$ as the 
descent generating polynomial for permutations in the Coxeter group $A_n$,
the group $\mathrm{Sym}(n+1)$ of all permutations on $n+1$ letters (see 
\cite{Foata2010, Sta97}). The notion of a descent can be extended to elements 
of all finite Coxeter groups as follows: for an element $\sigma$ of the Coxeter 
group $W$ the descents are exactly those generators $s$ of $W$ whose 
action on $\sigma$ reduces its length. In~\cite{Bre94}, Brenti used this 
interpretation of a descent to define the $W$-Eulerian polynomial, denoted 
$W(x)$, for any finite Coxeter group $W$. 

Brenti showed that many classical results about $A_n(x)$ hold for the other 
Eulerian polynomials as well.  In this paper we will investigate the remarkable 
property that $A_n(x)$ has only real roots,
a result due to Frobenius~\cite{Fro10}. Brenti proved the analogous result for type $B$, 
and checked by computer that it also held for the exceptional cases, but left 
type $D$---the only remaining case---as a conjecture. 

\begin{conj}[Conjecture 5.2 in~\cite{Bre94}]
\label{conj:brenti}
For every finite Coxeter group $W$, the descent generating polynomial $W(x)$ has only real roots.
\end{conj}

Dilks, Petersen, and Stembridge later extended the definition of Eulerian polynomials to 
include affine descents, and proposed the following companion to Brenti's conjecture.
\begin{conj}[Conjecture 4.1 in~\cite{DPS09}]
\label{conj:dps}
For every finite Weyl group $W$, the affine descent generating polynomial $\widetilde W(x)$ 
has only real roots.
\end{conj}
Again, this was not completely proved---$\widetilde{A}_n(x)$ and 
$\widetilde{C}_n(x)$ were shown to have only real roots,
the exceptional cases were verified, but the real-rootedness of the affine
Eulerian polynomials of types ${B}$ and ${D}$ remains an open problem.

In this paper, we build on the idea of real \emph{stability}---a generalization of the notion of
real-rootedness to multivariate polynomials. We combine this with simple recurrences for 
multivariate refinements of certain Eulerian polynomials to provide simple proofs of multivariate
generalizations of known real-rootedness results.  
Specifically, we give a general framework to show that the recurrence relations satisfied by multivariate 
$W$- and $\widetilde{W}$-Eulerian polynomials (for certain finite Coxeter groups $W$) 
are stability-preserving.  We then use properties of 
stability to show that this implies that the univariate counterparts of these polynomials are also stable, 
which is equivalent the statement that they have only real roots. 

The remainder of this paper is structured as follows.  In Section~\ref{sec:prelim}, we introduce notation, 
define the $W$-Eulerian and the affine $\widetilde{W}$-Eulerian polynomials for finite Coxeter groups 
and finite Weyl groups, respectively. We also review the required definitions and results related to 
real stability. For clarity and completeness, we 
begin in Section~\ref{sec:typeA} with a proof due to Br\"and\'en of the stability of the multivariate 
Eulerian polynomial of type $A$.  In Sections~\ref{sec:typeB} and~\ref{sec:signed}, we generalize 
this idea in several directions simultaneously, to type $B$ (signed permutations) and the 
generalized symmetric group (colored permutations), and also to
multiple $q$ variables. Sections~\ref{sec:typetildeA} and~\ref{sec:typetildeC} then address the affine
Eulerian polynomials for types $A$ and $C$.  The unresolved cases of Conjecture~\ref{conj:brenti} 
(type $D$) and Conjecture~\ref{conj:dps} (types $B$ and $D$) are examined within our 
multivariate framework in Section~\ref{sec:typeD}. We conclude with a discussion about the 
connection between our statistics with Catalan numbers, Motzkin paths and Laguerre polynomials.

\section{Preliminaries}
\label{sec:prelim}

We begin by introducing some notation.  For a positive integer $n$, let $[n]$ be the set $\{1, \dots, n\}$ 
and let $\mathbold{x}$ be the $n$-tuple $(x_1, \dots, x_n)$;  for example, 
$\mathbold{x}+\mathbold{y}=(x_1+y_1,\ldots,x_n+y_n)$.  For $\mathcal{T}$ a set (or multiset) with entries from $[n]$, 
we let $\mathbold{x}^\mathcal{T} = \prod_{i \in \mathcal{T}} x_i$; for example, 
$(\mathbold{x}+\mathbold{y})^{[n]} = \prod_{i=1}^n (x_i+y_i).$ 
The cardinality of $\mathcal{T}$ is written $|\mathcal{T}|$. The concatenation of $\mathbold{x}$ and 
$\mathbold{y}$ is denoted by $(\mathbold{x},\mathbold{y})$. 
We apply a function $f$ of $n$ variables to an $n$-tuple $\mathbold{x}$ by writing $f(\mathbold{x})=f(x_1, \dots, x_n)$.  
We will often deal with functions that have $\mathbold{x}$ and $\mathbold{y}$ as variables, and so we define the special symbol 
$\bpartial = \sum_{i=1}^n (\partial/\partial x_i+\partial/\partial y_i)$
as a shorthand for the sum of partial derivatives with respect to all $x_i$ and $y_i$ variables.

Finally, the theorems and propositions that are taken from previous works are clearly marked
by a reference (indicating the source); as far as we know, all other results are new. 

\subsection{$W$-Eulerian Polynomials}
\label{sec:eulpoly}

Let $S$ be a set of Coxeter generators, $m$ be a Coxeter matrix, and 
\[
W = \left\langle S : (s s')^{m(s,s')} = e,\; \text{ for } s,s'\in S, \;m(s,s')<\infty \right\rangle
\] 
be the corresponding Coxeter group (see~\cite{BB05}).  Given such a Coxeter system $(W,S)$ 
and $\sigma \in W$, we denote by $\ell_W(\sigma)$ the length of $\sigma$ in $W$ with respect to $S$.

\begin{defn}
For $W$ a finite Coxeter group, with generator set $S =\{s_1, \dotsc, s_n \}$, the \emph{descent} set of $\sigma \in W$ is
\[
 \mathcal{D}_W(\sigma) = \left\{i \in [n]: \ell_W(\sigma s_i) < \ell_W(\sigma)\right\}.
\]
\end{defn}

\begin{defn}
  For $W$ a finite Coxeter group, the $W$-\emph{Eulerian polynomial} is the descent generating
  polynomial 
  \[
    W (x) = \sum_{\sigma \in W} x^{|\mathcal{D}_W(\sigma)|}.
  \]
\end{defn}

The above definitions were extended for a subset of finite Coxeter groups in \cite{DPS09} to include affine descents. 
\begin{defn}
For $W$ a finite Weyl group, the \emph{affine descent} set of $\sigma \in W$ is
\[
 \widetilde{\mathcal{D}}_{W}(\sigma) = \mathcal{D}_W(\sigma) \cup 
 \left\{0 : \ell_W(\sigma s_0) > \ell_W(\sigma)\right\},
\]
where $s_0$ is the reflection corresponding to the lowest root in the underlying crystallographic 
root system. See \cite{DPS09} for further details and the motivation behind this definition.
\end{defn}

\begin{defn}
  For $W$ a finite Weyl group, the $\widetilde {W}$-\emph{Eulerian polynomial} is the affine 
  descent generating polynomial (over the corresponding finite Weyl group $W$)
  \[
    \widetilde{W} (x) = \sum_{\sigma \in W} x^{|\widetilde{\mathcal{D}}_{W}(\sigma)|}.
  \]
\end{defn}

\subsection{Real Stable Polynomials}
\label{sec:stable}

We define real stability, which generalizes the 
notion of real-rootedness from real univariate polynomials to real multivariate polynomials.

Let $\mathcal{H}_+ = \{z \in \mathbb{C} : \Im(z)>0 \}$ denote the open upper complex 
half-plane and similarly let $\mathcal{H}_-= \{z \in \mathbb{C} : \Im(z)<0 \}$.

\begin{defn}
  A polynomial $f \in \mathbb{R}[\mathbold{x}]$ is (real) \emph{stable} if $f\equiv 0$ or for any 
  $\mathbold{z} \in \mathcal{H}_+^n$, $f(\mathbold{z}) \neq 0$.
\end{defn}

Note that a univariate polynomial $f(x) \in \mathbb{R}[x]$ has real roots if and only if it is stable.
Following \cite{Wag11}, we let $\mathfrak{S}_\mathbb{R}[\mathbold{x}]$ denote the set of stable 
polynomials in $\mathbb{R}[\mathbold{x}]$.

\medskip
In this paper, we have a fixed template for our proofs.  We argue by induction, first checking stability 
(by hand) for the base case.  Next, we establish recursive formulas of the following form 
\[
W_n = T\left(W_{n-1}\right),
\] 
where $W_n$ is the multivariate $W$-Eulerian polynomial of a group $W$ of 
rank $n$, and $T$ is some linear operator.  Finally, we show that the linear operator $T$ is 
stability-preserving, using the following theorem.

Recall that a polynomial $f(\mathbold{x})$ is multiaffine if the power of each indeterminate $x_i$ is 
at most one. For a set $\mathcal{P}$ of polynomials, let $\mathcal{P}^{MA}$ be the set of multiaffine polynomials in $\mathcal{P}$.

\begin{thm}[Part of Theorem 3.5 in~\cite{Wag11}]
\label{thm:stable}
Let $T: \mathbb{R}[\mathbold{x}]^{MA} \to \mathbb{R}[\mathbold{x}]$ be a linear operator acting on the
variables $\mathbold{x}$.  If the polynomial 
$T((\mathbold{x}+\mathbold{y})^{[n]}) \in 
\mathfrak{S}_\mathbb{R}[\mathbold{x},\mathbold{y}]$, 
then $T$ maps $\mathfrak{S}_\mathbb{R}[\mathbold{x}]^{MA}$ 
\emph{into} $\mathfrak{S}_\mathbb{R}[\mathbold{x}]$.
\end{thm}

Once the multivariate $W$-Eulerian polynomials are shown to be stable, we can then reduce
them to real stable univariate polynomials using the following operations.

\begin{lem}[Part of Lemma 2.4 in~\cite{Wag11}]
\label{lem:stable}
Given $i,j \in [n]$, the following operations preserve real stability of $f \in \mathbb{R}[\mathbold{x}]$:
\begin{enumerate}
         \item \emph{Differentiation:} $f \mapsto \partial f/\partial x_i.$
	\item \emph{Diagonalization:} $f \mapsto f|_{x_i=x_j}.$
	\item	\emph{Specialization:} for $a \in \mathbb{R}$, $f\mapsto f|_{x_i = a}.$
\end{enumerate}
\end{lem}

Finally, there is an easy-to-check condition for real stability that we will use to show that certain polynomials are not real stable. 

\begin{thm}[Theorem 5.6 in~\cite{Bra07}]
\label{thm:checkstable}
 Let $f \in \mathbb{R}[\mathbold{x}]^{MA}$.  Then $f$ is real stable if and only if for all $i,j \in [n]$ and 
 for all $\mathbold{a} \in \mathbb{R}^n$, 
 $\frac{\partial f}{\partial x_i}(\mathbold{a}) \frac{\partial f}{\partial x_j}(\mathbold{a})-\frac{\partial^2 f}{\partial x_i \partial x_j}(\mathbold{a}) f (\mathbold{a})\geq 0.$
\end{thm}

We note that most of these results have a complex counterpart, but for our purposes real stability
suffices---all the polynomials we consider have positive integer coefficients. For this
reason we will sometimes refer to real stable polynomials simply as stable polynomials.

\section{Stable $W$-Eulerian Polynomials}

For clarity and completeness, we begin with a proof of the stability of the multivariate Eulerian 
polynomial of type $A$ due to Br\"and\'en.

\subsection{Eulerian Polynomials of Type $A$}
\label{sec:typeA}

Let $A_n$ denote the Coxeter group of type $A$ of rank $n$.  We can regard $A_n$ as 
$\mathrm{Sym}(n+1)$, the group of all permutations on $[n+1]$ with generators 
$S = \{s_1, \ldots, s_{n}\}$, where $s_i$ is the transposition $(i,i+1)$ for $1 \leq i \leq n$.

\begin{prop}[Proposition 1.5.3 in~\cite{BB05}]
Given $\sigma=\sigma_1 \ldots \sigma_{n+1} \in A_n$ written in the one-line notation, the 
\emph{descent} set of $\sigma$ is
 \begin{equation*}
  \mathcal{D}_A(\sigma) = \left\{i \in [n] : \sigma_i > \sigma_{i+1}\right\}.
 \end{equation*}
\end{prop}

The following theorem is well-known. Frobenius already mentioned that it
follows from the recurrence these polynomials satisfy:
\[ A_n(x) = (n+1)xA_{n-1}(x) + (1-x)\left(xA_{n-1}(x)\right)'\,.
\] 

\begin{thm}[p. 829 of~\cite{Fro10}]
\label{thm:typeAreal}
 \begin{equation}
  A_n (x) = \sum_{\sigma \in A_n} x^{|\mathcal{D}_A (\sigma)|}
  \label{eq:An}
 \end{equation}
 has only real roots.
\end{thm}
Theorem~\ref{thm:typeAreal} can also be proven using Rolle's theorem (see proof of Theorem 1.34 in \cite{Bon12}).

\begin{defn} 
 Given $\sigma \in A_n$, define the type $A$ \emph{descent top} set to be
 \[
  \mathcal{DT}_A (\sigma) = \{\max(\sigma_i,\sigma_{i+1}) : 1 \leq i \leq n, \sigma_i > \sigma_{i+1}\},
 \]
 and similarly, let the type $A$ \emph{ascent top} set be
 \[
  \mathcal{AT}_A (\sigma) = \{\max(\sigma_i,\sigma_{i+1}): 1 \leq i \leq n, \sigma_i<\sigma_{i+1}\}.
 \]
\end{defn}

For example, when $\sigma = 3 1 4 5 2 \in A_4$, $\mathcal{DT}_A(\sigma) = \{3,5\} $ and $\mathcal{AT}_A(\sigma) = \{4,5\}.$ 
Note that the seemingly superfluous notation $\max(\sigma_i,\sigma_{i+1})$ simply reduces to $\sigma_i$ and $\sigma_{i+1}$ 
in the case of type $A$ descent top and ascent top sets, respectively. Its significance will become apparent when we introduce 
the type $B$ descent top and ascent top sets.

\begin{thm}[Br\"and\'en~\cite{Bra10}]
\label{thm:typeA}
 \begin{equation}
 \label{eq:typeApoly}
 A_n (\mathbold{x},\mathbold{y}) = \sum_{\sigma \in A_n} \mathbold{x}^{\mathcal{DT}_A (\sigma)} \mathbold{y}^{\mathcal{AT}_A (\sigma)}
 \end{equation}
 is stable.
\end{thm}
\begin{proof}
We proceed by induction.  Note that $A_0(x_1,y_1) = 1$ is stable.  By observing the effect of inserting $n+1$ into a 
permutation $\sigma \in A_{n-1}$ on the type $A$ ascent top and descent top sets, we obtain the following recursion. For $n>0$, we have
 
 \begin{equation}
  \label{eq:recurrencetypeA} A_{n}(\mathbold{x},\mathbold{y})=(x_{n+1}+y_{n+1})A_{n-1}(\mathbold{x},\mathbold{y})+x_{n+1}y_{n+1}\bpartial A_{n-1}(\mathbold{x},\mathbold{y}).
 \end{equation}

We remind the reader here that 
$\bpartial = \sum_{i=1}^{n} \left( \frac{\partial}{\partial x_i}+\frac{\partial}{\partial y_i} \right).$  
It is easy to check using Theorem~\ref{thm:stable} that the linear operator 
$T=(x_{n+1} + y_{n+1})+x_{n+1}y_{n+1}\bpartial$ is stability-preserving, because 
\begin{align*}
T((\mathbold{x}+\mathbold{u})^{[n]} &(\mathbold{y}+\mathbold{v})^{[n]}) = \\ 
&x_{n+1}y_{n+1} \underbrace{\left[\frac{1}{y_{n+1}}+\frac{1}{x_{n+1}}+ \sum_{i=1}^n \left(\frac{1}{x_i+u_i}+\frac{1}{y_i+v_i}\right)\right]}_{\text{In } \mathcal{H}_- \text{ when } \mathbold{x,y,u,v}\in\mathcal{H}^+}
(\mathbold{x}+\mathbold{u})^{[n]} (\mathbold{y}+\mathbold{v})^{[n]}
\end{align*}
is in $\mathfrak{S}_\mathbb{R}[\mathbold{x},\mathbold{y},\mathbold{u},\mathbold{v}]$. The result follows.
\end{proof}

Specializing the $y_i$ variables to $1$, it follows that

\begin{cor}
\[A_n(\mathbold{x}) = \sum_{\sigma \in A_n}\mathbold{x}^{\mathcal{DT}_A (\sigma)}\] is stable.
\end{cor}

Diagonalizing $\mathbold{x}$, we obtain

\begin{cor}
\[A_n(x) = \sum_{\sigma \in A_n} x^{|\mathcal{DT}_A (\sigma)|} = \sum_{\sigma \in A_n} x^{|\mathcal{D}_A (\sigma)|}\] is stable.
\end{cor}

Since $A_n(x)$ is univariate, this corollary is equivalent to the statement that $A_n(x)$ has only real roots (Theorem~\ref{thm:typeAreal}).

We refer to~\cite{HV11} for a proof of the stability of a (slightly different) 
multivariate refinement of the \emph{classical} Eulerian polynomials.  
That refinement has close connections with the affine Eulerian polynomial of 
type $C$, which we will address in more detail in 
Section~\ref{sec:typetildeC}.

\medskip
Next, we present our results. We start by defining the multivariate Eulerian 
polynomial of type $B$ and proving that it is stable.

\subsection{Eulerian Polynomials of Type $B$}
\label{sec:typeB}

Let $B_n$ denote the Coxeter group of type $B$ of rank $n$.  We regard $B_n$ 
as the group of all signed permutations on $[\pm n] = \{-n, \dots, -1, 1, \dots, n\}$ with generators 
$S = \{s_0, s_1, \ldots, s_{n-1}\}$, where $s_0$ is the transposition $\left(-1, 1\right)$ and 
$s_i = (i,i+1)$ for $1 \leq i \leq n-1$.  

Type $B$ descents have a simple combinatorial description that we will exploit.

\begin{prop}[Corollary 3.2 of~\cite{Bre94}, also Proposition 8.1.2 of~\cite{BB05}]
  \label{prop:typeBdescents}
 Given a signed permutation $\sigma = (\sigma_1, \dots, \sigma_n)\in B_n$, written in its window-notation, let
 \begin{equation*}
  \mathcal{D}_B(\sigma) = \left\{i \in [n] : \sigma_{i-1} > \sigma_{i}\right\},
 \end{equation*}
 where $\sigma_0 \buildrel\mathrm{def}\over= 0$.
\end{prop}

Analogously to type $A$, the type $B$ Eulerian polynomials have only real roots.

\begin{thm}[Brenti \cite{Bre94}]
\label{thm:typeBreal}
 \begin{equation}
  B_n (x) = \sum_{\sigma \in B_n} x^{|\mathcal{D}_B (\sigma)|}
 \end{equation}
 has only real roots.
\end{thm}

In~\cite{Bre94}, Brenti also introduced a ``$q$-analog'' of the 
(univariate) type $B$ Eulerian polynomials using the following signed permutation statistic.
For $\sigma \in B_n$, let
\begin{equation*}
 N(\sigma) = \left|\left\{i \in [n]: \sigma_i < 0\right\}\right|
\end{equation*}
denote the number of negative entries in the signed permutation $\sigma$.

\begin{thm}[Corollary 3.7 of \cite{Bre94}]
  \label{thm:Brenti}
  For $q\ge 0$,
   \begin{equation}
   \label{eq:typeBpoly}
    B_n(x;q) = \sum_{\sigma \in B_n} q^{N(\sigma)}x^{|\mathcal{D}_B(\sigma)|}.
   \end{equation}
   has only real roots.
\end{thm}

These $B_n(x;q)$ polynomials specialize to Eulerian 
polynomials $A_{n-1}(x)$ and $B_n(x)$---when $q=0$ and $q=1$, 
respectively---so that Theorem~\ref{thm:Brenti} simultaneously 
generalizes Theorems~\ref{thm:typeAreal} and~\ref{thm:typeBreal}.

\medskip
We will proceed in the same way that the multivariate Theorem~\ref{thm:typeA} extends 
the univariate Theorem~\ref{thm:typeAreal}.  Recall that the stability of the multivariate 
refinement of the type $A$ Eulerian polynomials in \eqref{eq:typeApoly} came from the careful 
choice of the statistic. Choosing the maximum of the two values $\sigma_i$ and $\sigma_{i+1}$ for
both ascents and descents allowed for the simple stability-preserving recursion.
We apply this idea to signed permutations in such a way that the definitions remain consistent with
the definitions for ordinary permutations.

\begin{defn} 
Given $\sigma \in B_n$, define the type $B$ \emph{descent top} set to be
\[
  \mathcal{DT}_B(\sigma) = \{ \max(|\sigma_i|,|\sigma_{i+1}|) : 0 \leq i \leq n-1, \,
  \sigma_i > \sigma_{i+1} \}.
 \]
Analogously, we define the type $B$ \emph{ascent top} set to be 
\[
  \mathcal{AT}_B(\sigma) = \{ \max(|\sigma_i|,|\sigma_{i+1}|) : 0 \leq i \leq n-1, \,
  \sigma_i < \sigma_{i+1} \}.
\]
\end{defn}

For example, when $\sigma = (3,1,-4, -5, 2) \in B_5$, $\mathcal{DT}_B(\sigma)=\{3,4,5\}$ 
and $\mathcal{AT}_B (\sigma)=\{3,5\}$.

\begin{thm}
  \label{thm:multivariateBrenti}
  For $q \ge 0$,
  
  \begin{equation}
\label{eq:typeBpolyhom}
B_n(\mathbold{x}, \mathbold{y}; q) = \sum_{\sigma \in B_n}q^{N(\sigma)} \mathbold{x}^{\mathcal{DT}_B(\sigma)} \mathbold{y}^{\mathcal{AT}_B(\sigma)}
 \end{equation}
   is stable.
\end{thm}

\begin{proof}
  As in the proof of Theorem~\ref{thm:typeA}, we proceed by induction. $B_1(x_1,y_1;q) = qx_1+y_1$ 
  is stable when $q\ge0$, which settles the base case.  By observing the effect on the ascent top and
  descent top sets of type $B$ of inserting $n+1$ or $-(n+1)$ into a signed permutation 
  $\sigma \in B_{n}$, we obtain the following recursion. For $n>0$, we have
\begin{equation}
  \label{eq:Bqrecurrence}
  B_{n+1}(\mathbold{x},\mathbold{y};q) = 
  (qx_{n+1} + y_{n+1})B_{n}(\mathbold{x},\mathbold{y};q) + (1+q)x_{n+1}y_{n+1}\bpartial B_{n}(\mathbold{x},\mathbold{y};q).
 \end{equation}

To complete the proof, we note that for a fixed $q\ge0$, the linear operator acting on the right hand side, 
$T =  (qx_n + y_n)+(1+q)x_ny_n\bpartial$ preserves stability by Theorem~\ref{thm:stable}, since 
\begin{align*}
T((\mathbold{x}+\mathbold{u})^{[n]}&(\mathbold{y}+\mathbold{v})^{[n]}) =  \\ & x_{n+1} y_{n+1} 
\left[\frac{q}{y_{n+1}} + \frac{1}{x_{n+1}} + 
\sum_{i=1}^{n}\left(\frac{1+q}{x_i+u_i} + \frac{1+q}{y_i+v_i}\right)\right] 
(\mathbold{x}+\mathbold{u})^{[n]}(\mathbold{y}+\mathbold{v})^{[n]}
\end{align*}
is in $\mathfrak{S}_\mathbb{R}[\mathbold{x},\mathbold{y},\mathbold{u},\mathbold{v}]$ whenever $q\ge 0$.
\end{proof}

Our theorem has some noteworthy consequences.
The value of $B_n(\mathbold{x},\mathbold{y};q)$ at $q=-1$ is immediate from the recursion.

\begin{cor}
 \label{cor:Bnminusone}
 \begin{equation*}
   B_n(\mathbold{x},\mathbold{y};-1)= (\mathbold{y}-\mathbold{x})^{[n]}
  \end{equation*}
\end{cor}

This refines the following theorem of Reiner.
\begin{cor}[Theorem 3.2 of \cite{Rei95}]
\[ \sum_{\sigma \in B_n} \delta(\sigma)x^{\mathcal{D}_B(\sigma)} = \prod_{i=1}^n (1-x_i),\]
where $\delta$ is a one-dimensional character of $B_n$ defined by $\delta(\sigma) = (-1)^{N(\sigma)}$. 
\end{cor}

If we set $q=1$, we obtain the analogue of Theorem~\ref{thm:typeA} for type $B$.
\begin{cor}
 \begin{equation*}
   B_n(\mathbold{x},\mathbold{y};1)=\sum_{\sigma \in B_n} \mathbold{x}^{\mathcal{DT}_B(\sigma)}\mathbold{y}^{\mathcal{AT}_B(\sigma)}
 \end{equation*}
 is stable.
\end{cor}

We would like to point out that when we plug in $q=0$ into \eqref{eq:typeBpolyhom}, we get a \emph{homogenized} polynomial that is 
not equal to the polynomial $A_{n-1}(\mathbold{x}, \mathbold{y})$ from \eqref{eq:typeApoly}, since their recursions differ.  
Rather, $B_n(\mathbold{x},\mathbold{y};0)$ is the permanent of the following $n \times n$ matrix 
$M=(m_{ij})$ considered in~\cite{BHVW11}. For $i,j \in [n]$, let $m_{ij} = x_i$, when $i < j$ and 
$m_{ij}=y_j$, otherwise.  When we expand the permanent by the last column, we obtain the recurrence in \eqref{eq:Bqrecurrence} 
with $q=0$ (see Lemma 3.3 in~\cite{BHVW11} for a proof).

\medskip
Specializing the $y$ variables in $B_n(\mathbold{x},\mathbold{y};q)$ to 1, it follows that

\begin{cor}
  For $q \ge 0$, \[B_n(\mathbold{x};q)=\sum_{\sigma \in B_n}q^{N(\sigma)} \mathbold{x}^{ \mathcal{DT}_B(\sigma)}\] is stable.
\end{cor}

Finally, observe that (the non-homogeneous) $B_n(\mathbold{x}; q)$ does reduce 
to $A_{n-1}(\mathbold{x})$ and $B_n(\mathbold{x})$---the multivariate Eulerian polynomial of 
type $A$ and type $B$---when $q=0$ and $q=1$, respectively.  Diagonalizing $\mathbold{x}$ 
in $B_n(\mathbold{x};q)$ yields the polynomial $B_n(x;q)$ defined in \eqref{eq:typeBpoly}.  
We therefore recover Theorem~\ref{thm:Brenti} as a corollary.

\begin{cor}
  For $q \ge 0$, $B_n(x; q)$ is stable.
\end{cor}

\subsection{Eulerian Polynomials for Colored Permutations}
\label{sec:signed}

Theorem~\ref{thm:multivariateBrenti} can be extended in two directions simultaneously: from signed permutations to colored permutations, and from a single $q$ variable to several.

Let $\mathbb{Z}_r$ denote the cyclic group of order $r$ with generator $\zeta$.  
We will take $\zeta$ to be an $r$th primitive root of unity. The wreath product 
$G_n^r = \mathbb{Z}_r \wr A_{n-1}$ is the semidirect product $(\mathbb{Z}_r)^{\times n} \rtimes A_{n-1}$. Its elements can 
be thought of as $\sigma = (\zeta^{e_1} \tau_1, \dots, \zeta^{e_n} \tau_n),$ where $e_i \in \{0, 1, \ldots, r-1\}$ and 
$\tau \in A_{n-1}$. 
The group $G_n^r$ is sometimes called the generalized symmetric group, since
$\mathcal{A}_{n-1} \cong \mathrm{Sym}(n)$. Its elements are also known as 
$r$-colored permutations, which reduce to signed permutations and ordinary 
permutations when $r=2$ and $r=1$, respectively. In other words, $B_n = G_n^2$ 
and $A_{n-1} = G_n^1$.

\begin{defn}
  Given $\sigma = (\zeta^{e_1} \tau_1, \dots, \zeta^{e_n} \tau_n) \in G_n^r$, let $\mathcal{N}(\sigma)$ 
  be the multiset in which each $i \in [n]$ appears $e_i$ times.  
\end{defn}
Note that for $\sigma \in B_n$, $|\mathcal{N}(\sigma)| = N(\sigma)$, the number of negative entries in 
$\sigma = (\sigma_1, \dots, \sigma_n)$.

We adopt the following total order on the elements of 
$(\mathbb{Z}_r \times [n]) \cup \{0\}$ (see \cite{Ass10,BZ11}, for example):
\begin{align*}
 \label{eq:ordering}
  \zeta^{r-1} n< \dots < \zeta n < \dots < \zeta^{r-1}2 < \dots <\zeta 2 <
   \zeta^{r-1} 1< \dots< \zeta 1 < 0  < 1 < 2 < \dots <n.
\end{align*}

While other total orders are also being used in the literature (e.g., 
\cite{Steingr94,Wag03,CM11}) our choice allows for similar stability-preserving recurrences as in the previous cases.
Using this ordering, the definitions of descent top set 
and ascent top set all extend verbatim from $B_n$ to $G_n^r$.  
We shall use $\mathcal{DT}_r(\sigma)$ and $\mathcal{AT}_r(\sigma)$ to denote
them for $\sigma$ in $G_n^r$.
For example, when $\sigma = (3,\zeta^21,\zeta^24,\zeta^45,\zeta 2) \in G_n^5$, then we have $0 < 3 > \zeta^21 > \zeta^24 > \zeta^45 < \zeta 2$ 
and hence, $\mathcal{DT}_{5}(\sigma) = \{3,4,5\}$,
and $\mathcal{AT}_{5}(\sigma) = \{3,5\}$.

\medskip
Br\"and\'en generalized Brenti's $B_n (x;q)$ polynomial, defined in \eqref{eq:typeBpoly}, to
multiple $q$ variables, and proved the following.

\begin{thm}[Corollary 6.5 in~\cite{Bra06}]
 Let $\mathbold{q}=(q_1,...,q_n)$. If $q_i \geq 0$, for $1 \le i \le n$, then
 \begin{equation}
 \label{eq:typeBqs}
  B_n(x; \mathbold{q}) = \sum_{\sigma \in B_n} \mathbold{q}^{\mathcal{N}(\sigma)} x^{|\mathcal{D}_B(\sigma)|}.
 \end{equation}
has only simple real roots. 
\end{thm}

Next, we extend this result simultaneously to $G_n^r$ and to multiple $x$ variables.

\begin{thm}

If $q_i \geq 0$, for all $1\le i\le n$, then the multivariate 
$\mathbold{q}$-Eulerian polynomial for the generalized symmetric group, 
$G_n^r$, defined as

 \begin{equation}
  \label{eq:qrefinement}
  G_n^r(\mathbold{x}, \mathbold{y}; \mathbold{q}) = \sum_{\sigma \in G_n^r} 
  \mathbold{q}^{\mathcal{N}(\sigma)} \mathbold{x}^{\mathcal{DT}_r(\sigma)}
  \mathbold{y}^{\mathcal{AT}_r(\sigma)}
\end{equation}
is stable.

\end{thm}

\begin{proof}
  $G_1^r(x_1,y_1;q_1) = (q_1+ \dots + q_1^{r-1}) x_1 + y_1$ is clearly stable when $q_1 \ge 0$. 
  The theorem follows immediately from the following recursion. For $n > 1$, 
  \[
      G_n^r(\mathbold{x},\mathbold{y};\mathbold{q}) = 
      \left[((q_n+\dots + q_n^{r-1})x_n + y_n) +  
      (1+ \dots + q_n^{r-1})x_ny_n\bpartial\right] G_{n-1}^r(\mathbold{x},\mathbold{y};\mathbold{q}).  \qedhere
  \]
\end{proof}

As a consequence, we obtain a generalization of Corollary \ref{cor:Bnminusone} to $G_n^r$.

\begin{cor}
 Let $r \ge 2$. For an $r$th root of unity, $\zeta\neq 1$, we have
 \begin{equation*}
   G_n^r(\mathbold{x}, \mathbold{y}; \underbrace{\zeta, \dots, \zeta}_{n}) = (\mathbold{y}-\mathbold{x})^{[n]}.
 \end{equation*}
\end{cor}

Letting $r=2$ also generalizes Theorem~\ref{thm:multivariateBrenti} to multiple $q$ variables.
\begin{cor}
 If $q_i \geq 0$, for all $1\le i\le n$, then $B_n(\mathbold{x}, \mathbold{y}; \mathbold{q})$ is stable.
\end{cor}

Diagonalizing $\mathbold{q}$ gives us a result for $G_n^r$ with a single $q$ variable.
\begin{cor}
 If $q \geq 0$, then $G_n^r(\mathbold{x}, \mathbold{y}; q)$ is stable.
\end{cor}

By diagonalizing $\mathbold{x}$ and specializing $y_i$ to $1$ for all $1\le i\le n$, we obtain a result of 
Steingr{\'i}msson.
\begin{cor}[Theorem 17 of \cite{Steingr94}] 
  \[G_n^r(x) = \sum_{\sigma \in G_n^r} x^{|\mathcal{D}_r(\sigma)|}\] 
  has only real roots.
 \end{cor}

\section{Stable $\widetilde{W}$-Eulerian Polynomials}
\label{sec:affine}

Dilks, Petersen and Stembridge studied Eulerian-like polynomials associated
to affine Weyl groups. They defined the so-called ``affine''  $\widetilde{W}$-Eulerian polynomials
as the ``affine descent''-generating polynomials over the corresponding finite Weyl group. 
In~\cite{DPS09}, they showed that the (univariate) $\widetilde{W}$-Eulerian polynomials have 
only real roots for types ${A}$ and ${C}$, and also for the exceptional types.
We strengthen these results for types ${A}$ and ${C}$ by giving multivariate stable 
refinements of these polynomials as well.

\subsection{Affine Eulerian Polynomials of Type ${A}$}
\label{sec:typetildeA}

Let $A_n$ denote the Coxeter group of type $A$ of rank $n$. The affine descents of 
type $A$ contain the (ordinary) descents of type $A$ and an extra ``affine'' 
descent at $0$ if and only if, 
$\sigma_{n+1} > \sigma_1$, where $\sigma = (\sigma_1, \dots, \sigma_{n+1}) \in {A}_{n}.$ 
Formally,
\[
  \widetilde{\mathcal{D}}_{A}(\sigma) = \mathcal{D}_{A}(\sigma) \cup \{0: \sigma_{n+1} > \sigma_1\}.
\]
See Section 5.1 in~\cite{DPS09} for further details.

The definitions of descent top and ascent top sets for type $A$ can be extended in the obvious way. For $\sigma \in A_n$,
\begin{align*}
  \widetilde{\mathcal{DT}}_{A}(\sigma) &= \mathcal{DT}_{A}(\sigma) 
  \cup \{\sigma_{n+1} : \sigma_{n+1} > \sigma_1\},\\ 
  \widetilde{\mathcal{AT}}_{A}(\sigma) &= \mathcal{AT}_{A}(\sigma) 
  \cup \{\sigma_{1} : \sigma_{n+1} < \sigma_1\}
\end{align*}
and we obtain the following result.

\begin{thm} 
  \[
    \widetilde A_{n}(\mathbold{x},\mathbold{y}) =
    \sum_{\sigma \in A_n} \mathbold{x}^{\widetilde{\mathcal{DT}}_{A} (\sigma)} 
    \mathbold{y}^{\widetilde{\mathcal{AT}}_{A} (\sigma)}
  \] is stable.
\end{thm}

\begin{proof}
  This statement is immediate once we establish the Lemma~\ref{lem:FulmanPetersenA}. 
Stability follows, since $A_n(\mathbold{x},\mathbold{y})$ is stable and the operator on the right-hand side 
is clearly stability-preserving.
\end{proof}

\begin{lem} For $n > 0$, we have
\label{lem:FulmanPetersenA}
  \[
    \widetilde{A}_{n}(\mathbold{x},\mathbold{y}) = (n+1) x_{n+1}y_{n+1} {A}_{n-1}(\mathbold{x},\mathbold{y}).
  \]
\end{lem}
\begin{proof}
  Consider a permutation $\sigma \in A_{n-1}$ with ascent top set $\mathcal{A}$ and 
  descent top set $\mathcal{D}$. We will modify it to obtain a permutation in $A_{n}$ with
  affine ascent top set $\mathcal{A} \cup \{n+1\}$ and affine descent top set 
  $\mathcal{D} \cup \{n+1\}$. 
  Append $(n+1)$ to the end of $\sigma$ and pick a cyclic rotation of the newly obtained 
  permutation. The new permutation will have the same affine ascent top and affine descent 
  top sets as the ascent top and descent top sets $\sigma$ had and in addition it will have
  $(n+1)$ both as an affine ascent top and as an affine descent top. To conclude the proof, 
  note that there are exactly $n+1$ cyclic rotations. This is essentially a refinement of 
  the proof of Proposition 1.1 of \cite{Pet05}.
\end{proof}

By diagonalizing $\mathbold{x}$ and specializing $\mathbold{y}$ to 1, Lemma \ref{lem:FulmanPetersenA} 
reduces to an identity discovered by Fulman (Corollary 1 in \cite{Ful00}) and
we recover a result of Dilks, Petersen, and Stembridge:
\begin{cor}[see Section 4 of \cite{DPS09}]
  \[
    \widetilde{A}_{n}(x) = \sum_{\sigma \in A_n} x^{|\widetilde{\mathcal{D}}_{A}(\sigma)|}
  \] has only real roots.
\end{cor}

One can construct a recurrence from Lemma \ref{lem:FulmanPetersenA} for the affine Eulerian
polynomials $\widetilde{A}_{n}(x)$ as well, but this recurrence will not preserve stability.
\subsection{Affine Eulerian Polynomials of Type ${C}$}
\label{sec:typetildeC}
Let $C_n$ denote the Coxeter group of type $C$ of rank $n$.
Affine descents of type $C$ consist of the ordinary descent set of type 
$C$, which coincides with the descent set of type $B$ 
(see Proposition~\ref{prop:typeBdescents} for type $B$ descents) and an 
extra ``affine'' descent at $0$ when $\sigma_n > 0$. Formally,
\[
  \widetilde{\mathcal{D}}_{C}(\sigma) = \mathcal{D}_{C}(\sigma) \cup \{0: \sigma_n > 0\}.\\
\]
See Section 5.2 of~\cite{DPS09} for further details.

As in type ${A}$, the definition of the type ${C}$ affine ascent and descent top sets can 
be adapted from those of type $C$ (equivalently, type $B$). For $\sigma \in C_n$, let

\begin{align*}
 \widetilde{\mathcal{DT}}_{C}(\sigma) &= \mathcal{DT}_{B}(\sigma) 
 \cup \{\sigma_n : \sigma_n > 0\},\\
  \widetilde{\mathcal{AT}}_{C}(\sigma) &= \mathcal{AT}_{B}(\sigma) 
 \cup \{\sigma_n : \sigma_n < 0\}.
\end{align*}

\begin{thm}
\[\widetilde C_{n}(\mathbold{x},\mathbold{y}) = \sum_{\sigma\in C_n} 
\mathbold{x}^{\widetilde{\mathcal{DT}}_{C}(\sigma)} 
\mathbold{y}^{\widetilde{\mathcal{AT}}_{C}(\sigma)}\]
is stable.
\end{thm} 
\begin{proof}
$\widetilde{C}_1(x_1,y_1) = 2x_1y_1$ and the following recurrence holds for $n>1$:
\[
  \widetilde{C}_{n}(\mathbold{x},\mathbold{y}) = 2x_ny_n \bpartial \widetilde{C}_{n-1}(\mathbold{x},\mathbold{y}).
  \qedhere
\]
\end{proof}
We note that a very similar recurrence also appeared in \cite{HV11} (without the factor of two and
with a different initial value), in connection with stable Eulerian polynomials over Stirling permutations.

There is also a direct connection between the polynomials 
$\widetilde{C}_n(\mathbold{x},\mathbold{y})$ and $A_n(\mathbold{x},\mathbold{y})$.
\begin{prop}
  \[
    \widetilde{C}_n(x_1, \dots, x_n, y_1, \dots, y_n) = 
    2^n x_ny_n A_{n-1}(x_0, \dots, x_{n-1}, y_0, \dots, y_{n-1}). 
  \]
\end{prop}
\begin{proof}
Follows by a similar argument as Lemma \ref{lem:FulmanPetersenA}.
\end{proof}

Once again, this gives a refinement of the univariate identity by Fulman \cite{Ful01}. We mention
that multivariate refinements (different from the above) and respective refinements for the
identities for $\widetilde{A}_n(\mathbold{x},\mathbold{y})$ and 
$\widetilde{C}_n(\mathbold{x},\mathbold{y})$
have appeared in \cite{DPS09}.

\section{Towards Stable Refinements of $D(x)$, $\widetilde{B}(x)$, $\widetilde{D}(x)$}
\label{sec:typeD}

\subsection{Eulerian Polynomials of Type $D$.}
Let $D_n$ denote the Coxeter group of type $D$ of rank $n$.  Recall that $D_n$ can be thought of as the
(order 2) subgroup of $B_n$ consisting of all the signed permutations with an even number of negative entries.  
Specifically, $D_n$ has generators $S = \{s_0, s_1, \ldots, s_{n}\}$, where $s_0=(-2,1)(-1,2)$ and $s_i = (i,i+1)$ for $1 \leq i \leq n$.

\begin{prop}[Proposition 8.22 of~\cite{BB05}]
 Given $\sigma =  (\sigma_1, \dotsc, \sigma_n) \in D_n$ written in its window notation,
 \begin{equation*}
  \mathcal{D}_D(\sigma) = \{i \in [n] : \sigma_{i-1} > \sigma_{i}\},
 \end{equation*}
  where $\sigma_0 \buildrel\mathrm{def}\over= -\sigma_2$. 
\end{prop}

Based on the real-rootedness results for types $A$, $B$, and the exceptional types (and some
computer evidence) Brenti conjectured the following.
\begin{conj}[Conjecture 5.1 in~\cite{Bre94}]
\label{conj:Brenti}
  The type $D$ Eulerian polynomial has only real roots. 
\end{conj}
This was---and still is---the remaining unproven part of Conjecture~\ref{conj:brenti}.
Unfortunately, extending the type $B$ ascent top and descent top definitions in the na\"ive way does not work. If we let 
\begin{align}
  \label{eq:typeDnaive}
  \mathcal{DT}_D(\sigma) &= 
  \{ \max(|\sigma_i|,|\sigma_{i+1}|) : 0 \leq i \leq n-1, \, \sigma_i > \sigma_{i+1}\}, \text{ and } \\ 
  \notag \mathcal{AT}_D(\sigma) &= 
  \{ \max(|\sigma_i|,|\sigma_{i+1}|) : 0 \leq i \leq n-1, \, \sigma_i < \sigma_{i+1} \},
\end{align}
then the type $D$ Eulerian polynomial is not multiaffine (for example, the monomial corresponding to
$\sigma = 123$ is $y_2^2 y_3$). Furthermore, it fails to be stable. For $n=3$, we have 
\[
  D^*_3(\mathbold{x}, \mathbold{y}) = x_2^2 x_3 + 2 x_2 x_3 y_2 + x_3 y_2^2 + x_2^2 y_3 + 4 x_2 x_3 y_3 + 4 x_3^2 y_3 + 2 x_2 y_2 y_3 + 4 x_3 y_2 y_3 + y_2^2 y_3 + 4 x_3 y_3^2.
\]
When $y_2=y_3=x_3=2+i$, and $x_2=(-1 + 2i) (2i + \sqrt{3}) \approx -5.7321+1.4641i$ 
(which are all in the upper half plane), $D^*_3(\mathbold{x},\mathbold{y})=0$. (We use the 
$D^*(\mathbold{x},\mathbold{y})$ notation for the multivariate Eulerian polynomial of type $D$ to avoid confusion with a different multivariate generalization
which will appear later on.)

\medskip
Br\"and\'en used the refinement given in \eqref{eq:typeBqs} to prove an intriguing result about 
enumerating type $B$ descents over type $D$ permutations.

\begin{thm}[Corollary 6.10 in~\cite{Bra06}]
\begin{equation*}
 \sum_{\sigma\in D_n} x^{|\mathcal{D}_B(\sigma)|}
\end{equation*}
has only real roots.
\end{thm}

It would be interesting to find a multivariate analog of this result.  Again, the straightforward application
of our method with the descent top set definition in \eqref{eq:typeDnaive} results in a non-stable polynomial already for $n=3$.

Table~\ref{table:D3} gives the type $B$ and type $D$ descents for permutations in $D_3$.  As demonstrated above, our choice of $X = x_{|\sigma_0|}$ and $Y = y_{|\sigma_0|}$ does not result in stable polynomials, though it is possible that some other choice will.
\begin{center}
\begin{table}[h]
\begin{tabular}{|c|c|c|c|c|c|c|c|c|c|c|c|}
\hline
$\sigma$ & $x^{\mathcal{DT}_B}$ & $x^{\mathcal{DT}_D}$ & $\sigma$ & $x^{\mathcal{DT}_B}$ & $x^{\mathcal{DT}_D}$ &$\sigma$ & $x^{\mathcal{DT}_B}$ & $x^{\mathcal{DT}_D}$ & $\sigma$ & $x^{\mathcal{DT}_B}$ & $x^{\mathcal{DT}_D}$\\ 

$123$ & $Yy_2y_3$ & $Yy_2y_3$ & $\bar{1}\bar{2}3$ & $Xx_2y_3$ & $Xy_2y_3$ &
$\bar{1}2\bar{3}$ & $Xy_2x_3$ & $Yy_2x_3$ & $1\bar{2}\bar{3}$ & $Yy_2x_3$ & $Xy_2x_3$ \\

$132$ & $Yy_3x_3$ & $Yy_3x_3$ & $\bar{1}\bar{3}2$ & $Xx_3y_3$ & $Xx_3y_3$ &
$\bar{1}3\bar{2}$ & $Xy_3x_3$ & $Yy_3x_3$ & $1\bar{3}\bar{2}$ & $Yx_3y_3$ & $Xx_3y_3$ \\

$213$ & $Yx_2y_3$ & $Yx_2y_3$ & $\bar{2}\bar{1}3$ & $Xy_2y_3$ & $Xy_2y_3$ &
$\bar{2}1\bar{3}$ & $Xy_2x_3$ & $Xy_2x_3$ & $2\bar{1}\bar{3}$ & $Yx_2x_3$ & $Yx_2x_3$ \\

$231$ & $Yy_3x_3$ & $Yy_3x_3$ & $\bar{2}\bar{3}1$ & $Xx_3y_3$ & $Xx_3y_3$ &
$\bar{2}3\bar{1}$ & $Xy_3x_3$ & $Yy_3x_3$ & $2\bar{3}\bar{1}$ & $Yx_3y_3$ & $Xx_3y_3$ \\

$312$ & $Yx_3y_2$ & $Yx_3y_2$ & $\bar{3}\bar{1}2$ & $Xy_3y_2$ & $Xy_3y_2$ &
$\bar{3}1\bar{2}$ & $Xy_3x_2$ & $Yy_3x_2$ & $3\bar{1}\bar{2}$ & $Yx_3x_2$ & $Xx_3x_2$ \\

$321$ & $Yx_3x_2$ & $Yx_3x_2$ & $\bar{3}\bar{2}1$ & $Xy_3y_2$ & $Xy_3y_2$ &
$\bar{3}2\bar{1}$ & $Xy_3x_2$ & $Xy_3x_2$ & $3\bar{2}\bar{1}$ & $Yx_3y_2$ & $Yx_3y_2$ \\
\hline
\end{tabular}
\newline\vspace{1mm} 
\caption{Type $B$ and type $D$ descents over the permutations in $D_3$.}
\label{table:D3}
\end{table}
\end{center}

Another result related to resolving Conjecture~\ref{conj:Brenti} came from~\cite{Cho03}, in which 
Chow found a recurrence for $D_n(x)$.  Sadly, the resulting formula is much more complicated 
than its counterparts in types $A$ and $B$, and does not seem amenable to a multivariate 
generalization.  We reproduce the recurrence here, fixing a typo from the original paper (the term 
in the box was mistakenly written with a minus sign).

\begin{thm}[cf.~Theorem 5.3 in~\cite{Cho03}] Let $D_{-1}(x) = D_0(x) = D_1(x) = 1$ and for $n\ge 0$

\begin{align*}
 D_{n+2}(x) &= \left(n (1 + 5 x) + 4 x\right) D_{n + 1}(x) \\
  &\quad +4 x (1 - x) D'_{n + 1} (x) \\
  &\quad +\left((1 - x)^2 - n (1 + 3 x)^2 - 4 n (n - 1) x (1 + 2 x)\right) D_n(x)\\
  &\quad -\left(4 n x (1 - x) (1 + 3 x) + 4 x (1 - x)^2\right) D'_n(x)\\
  &\quad -4 x^2 (1 - x)^2 D''_n (x)\\
  &\quad +\left(2 n (n - 1) x (3 + 2 x + 3 x^2) \text{\fbox{$+\, 4 n (n - 1) (n - 2) x^2 (1 + x)$}}\,\right) D_{n - 1} (x)\\
  &\quad +\left(2 n x (1 - x)^2 (3 + x) +  8 n (n - 1) x^2 (1 - x) (1 + x)\right)  D'_{n - 1}(x)\\
  &\quad +4 n x^2 (1 - x)^2 (1 + x) D''_{n - 1}(x).
\end{align*}   

\end{thm}

Encouraged by the simplicity of our methods for types $A$ and $B$, we propose a new line of attack.  
Stembridge showed that the Eulerian polynomials of types $A$, $B$, and $D$ are related via the following identity. 
This result was discovered independently by Brenti (see Corollary~4.8 and Theorem~4.7 in \cite{Bre94} for a ``$q$-analog''), 
who also points out that if follows from Theorem~4.2 of \cite{Rei95}.

\begin{thm}[Lemma 9.1 of \cite{Ste94}]
  For $n \ge 2$, 
  \begin{equation}
    \label{eq:Stembridge}
    D_n(x) = B_n(x) - n2^{n-1}xA_{n-2}(x),
  \end{equation}
  where $A_{n}(x)$, defined in \eqref{eq:An}, is the descent generating function in 
  $A_n \cong \mathrm{Sym}(n+1)$. 
\end{thm}

By replacing the univariate polynomials $B_n(x)$ and $A_n(x)$ in \eqref{eq:Stembridge} by their stable 
multivariate generalizations given in \eqref{eq:typeApoly} and \eqref{eq:typeBpolyhom}, respectively, we 
obtain the following multivariate refinement of $D_n(x)$:
\begin{equation}
 \label{eq:typeD}
  D_n(\mathbold{x}, \mathbold{y}) = B_n(\mathbold{x}, \mathbold{y}; 1) - n 2^{n-1} x_n y_n A_{n-2}(\mathbold{x}, \mathbold{y}).
\end{equation}
For example, when $n=2$, we obtain
\[
  D_2(x_1, x_2, y_1, y_2) = B_2(x_1, x_2, y_1, y_2;1) - 4x_2y_2A_0(x_1,y_1) = (x_1+y_1)(x_2+y_2),
\]
and when $n=3$, the polynomial is
\begin{align*}
  D_3(\mathbold{x},\mathbold{y})  & =  x_1 x_2 x_3+x_2 x_3 y_1+x_1 x_3 y_2+x_3 y_1 y_2+x_1 x_2 y_3+x_2 y_1 y_3+x_1 y_2 y_3+y_1 y_2 y_3 \\
  &\quad+ 4(x_2 x_3 y_2+x_1 x_3 y_3+ x_2 y_2 y_3+ x_3 y_1 y_3).
\end{align*}

These polynomials fail to be stable, even for $n=3$. This follows from Theorem~\ref{thm:checkstable}, 
since specializing the $y$ variables to 1 gives 
\[
  D_3 (\mathbold{x}) = 1 + x_1 + x_1 x_2 + x_1 x_2 x_3+5 (x_2 +x_3 + x_1 x_3 + x_2 x_3)
\] 
and
\[
 \frac{\partial D_3(\mathbold{x})}{\partial x_1} \cdot \frac{\partial D_3(\mathbold{x})}{\partial x_3}-
 \frac{\partial^2 D_3(\mathbold{x})}{\partial x_1 \partial x_3} \cdot D_3 (\mathbold{x}) = -16x_2,
\]
which fails to be nonnegative for $x_2 >0$.  It is possible that a different refinement would 
result in a family of stable polynomials, from which real-rootedness would follow.

It does seem to be the case, however, that the coefficients remain non-negative.

\begin{conj}
  $D_n(\mathbold{x}, \mathbold{y}) \in \mathbb{N}[\mathbold{x},\mathbold{y}]$ for all $n\ge 2$.
\end{conj}

We have verified this conjecture by computer for $n \leq 11$.  If true, then \eqref{eq:typeD} suggests a 
refinement of the descent statistic of type $D$ which in turn could lead to a new recursion for the type $D$ 
Eulerian polynomials.

\subsection{Affine Eulerian Polynomials of Types $B$ and $D$}
\label{sec:typetildeBD}

Dilks, Petersen, and Stembridge noted that the only missing cases from Conjecture~\ref{conj:dps} are types $B$ and
$D$.  They considered multivariate refinements of these polynomials, but they indexed
the variables by descents and not by descent tops. It was noted in~\cite{HV11} that polynomials 
indexed by the descents fail to be stable for type $A$.

Two new identities relating ordinary and affine Eulerian polynomials also appeared
in \cite{DPS09}. 
\begin{prop}[Proposition 6.1 in~\cite{DPS09}]
  For $n \ge 2$,
  \[
    2\widetilde{C}_n(x) = \widetilde{B}_n(x) + 2nxC_{n-1}(x).
  \]
\end{prop}
This identity could be used to obtain a multivariate refinement for the unsettled type ${B}$ case 
in a similar way as suggested for type $D$ above. Let 
\begin{equation}
\label{eq:affineCpolynomial}
  \widetilde{B}_n(\mathbold{x},\mathbold{y}) = 
  2\widetilde{C}_n(\mathbold{x},\mathbold{y}) - 2nx_ny_nB_{n-1}(\mathbold{x},\mathbold{y};1)\,.
\end{equation}

In contrast to the type $D$ case, these polynomials turn out to be stable for $n\le 4$ 
(we have verified this by computer calculations). Therefore, we suggest to investigate 
whether this is a stable multivariate of the affine Eulerian polynomial of type $B$.
\begin{conj}
  $\widetilde{B}_n(\mathbold{x}, \mathbold{y}) \in \mathfrak{S}_{\mathbb{R}}[\mathbold{x},\mathbold{y}]$, 
  for $n\ge 2$.
\end{conj}
Also, this multivariate refinement seems to be monomial positive. 
\begin{conj}
  $\widetilde{B}_n(\mathbold{x}, \mathbold{y}) \in \mathbb{N}[\mathbold{x},\mathbold{y}]$, for $n\ge 2$.
\end{conj}

Finally, for sake of completeness, we give another identity by Dilks, Petersen and Stembridge which
relates the following polynomials.
\begin{prop}[Proposition 6.2 in \cite{DPS09}] For $n \ge 3$, 
  \begin{equation}
    \widetilde{B}_n(x) = \widetilde{D}_n(x) + 2nxD_{n-1}(x).
  \end{equation}
\end{prop} 
This identity might be helpful for finding a multivariate refinement of the affine Eulerian polynomial of type $D$.

\section{Conclusion and further remarks on the statistics $\mathcal{AT}, \mathcal{DT}$}

In this paper, we have extended the stability results for the multivariate type $A$ 
Eulerian polynomial to Eulerian polynomials (and affine Eulerian polynomials) of some 
other Coxeter groups and the generalized permutation group $G_{n,r}$. 
A crucial step in our proofs was to find a suitable generalization of the descent top 
and ascent top statistics to these groups.

The descent statistic $\{i : \sigma_i > \sigma_{i+1}\}$ for permutations has a long history, going back 
to the works of Carlitz and Riordan. Recall that for a permutation of length $n$ there are $2^{n-1}$ possible descent 
(and ascent) sets, as each position $i \in [n-1]$ can either be a descent or an ascent. In particular, for any subset 
$S \subset [n-1]$ there is a permutation whose descent set is $S$. Finding the 
number of permutations with a given descent (and ascent) set is a classical problem---one uses a standard 
inclusion-exclusion argument (see Example 2.2.4 in \cite{Sta97}).

We now find the possible descent and ascent top sets (of type $A$), which corresponds to the possible monomials 
in $A_n (\mathbold{x},\mathbold{y})$. (From here on, $\mathcal{DT}$ and $\mathcal{AT}$ refers
to the descent top and ascent top sets of type $A$.)
\begin{prop} 
  Given $(\mathcal{DT},\mathcal{AT}) \subset [n]\times[n]$ define 
  $\mathcal{P} = \mathcal{DT}\cap\mathcal{AT}$ and 
  $\mathcal{V} = [n]\setminus(\mathcal{DT}\cup\mathcal{AT})$. The pair $(\mathcal{DT},\mathcal{AT})$ 
  is the descent top and ascent top set of some permutation of $[n]$ if and only if, 
  for all $i=1,\ldots,n$, $|[i] \cap \mathcal{V}| > |[i] \cap \mathcal{P}|$.
\end{prop}
\begin{proof}
  Clear from the recurrence given in \eqref{eq:recurrencetypeA}.
\end{proof}

In fact, the $(\mathcal{DT}, \mathcal{AT})$ pair of statistics turns out to be equivalent to a 
well-known statistic studied by Fran\c{c}on and Viennot \cite{FV79}, called ``type''---the quadruple of 
statistics consisting of the set of peaks, valleys, double descents, and double ascents. In particular, 
the set of peaks $\{\sigma_i : \sigma_{i-1} < \sigma_i > \sigma_{i+1}\}$ is $\mathcal{P}$, defined above, 
and similarly, the set of valleys $\{\sigma_i : \sigma_{i-1} > \sigma_i < \sigma_{i+1}\}$ is $\mathcal{V}$.
For the set of double descents $\mathcal{DD}:=\{\sigma_i : \sigma_{i-1} > \sigma_i > \sigma_{i+1}\}$
we have that $\mathcal{DD} = \mathcal{DT}\setminus\mathcal{AT}$ (and similarly for the double 
ascents we have $\mathcal{DA} = \mathcal{AT}\setminus\mathcal{DT}$).

In light of the above, the following theorem should not be a surprise.
\begin{thm}
  The number of distinct monomials in the expansion of $A_{n-1}(\mathbold{x}, \mathbold{y})$ 
  is counted by $C_n = \frac{1}{n+1}\binom{2n}{n}$, the $n$th Catalan number.
\end{thm}
\begin{proof}
  We give a bijection from $(\mathcal{DT},\mathcal{AT}) \in [n]\times[n]$ to 2-colored Motzkin paths 
  of length $n-1$. 
  A \emph{2-colored Motzkin path} of length $n-1$ is a lattice path of $\mathbb{N}^2$ running from $(0,0)$ 
  to $(n-1,0)$ that never goes below the $x$-axis and whose allowed steps are:
\begin{itemize}
    \item $NE$ steps $(1,1)$,
    \item $SE$ steps $(1,-1)$, and
    \item Two different colors of $E$ steps $(1,0)$, denoted $\overline{E}$ and $\underline{E}.$
\end{itemize}
  For $j = 2,\ldots,n$, let the $(j-1)$st step of the Motzkin path be $SE$, $NE$, 
  $\overline{E}$, or $\underline{E}$ depending on which set ($\mathcal{P}, \mathcal{V}, 
  \mathcal{DD}$, or $\mathcal{DA}$, respectively) $j$ belongs to. The 2-colored Motzkin paths of 
  length $n-1$ are known to be counted by the Catalan numbers $C_n$ (see, for example, \cite{Sta11}).
\end{proof}

Exploring this connection, we can also give a formula for the coefficient on a particular monomial 
$\mathbold{x}^{\mathcal{DT}} \mathbold{y}^{\mathcal{AT}}$---that is, the number of permutations with 
a given descent and ascent top set.  
Following Viennot, we define a valuation on these 2-colored Motzkin paths.  $NE$ and $SE$ steps 
\emph{starting} at height $k\geq0$ are given a weight $a_k=c_k=k+1$.  We give $\overline{E}$ steps 
a weight of $\overline{b}_k=k+1$ and $\underline{E}$ steps a weight of $\underline{b}_k=k+1$, which 
is equivalent to forgetting that we had two different types of $E$ steps and weighting an (uncolored) 
east step by $b_k=2k+2$.
Viennot's theory of Laguerre histories~\cite{Vie84}, which gives a combinatorial interpretation for 
the moments of the Laguerre polynomials, allows us to then recover the permutations with given descent and ascent top sets. 
In particular, the number of permutations associated to a path is simply the product of the weights
 of the steps in the path.

Similar reasoning works for type $B$ as well.

\begin{thm}
  The number of distinct monomials 
  in the expansion of $B_{n}(\mathbold{x}, \mathbold{y};1)$ is counted by $C_n$, the $n$th Catalan number.
\end{thm}

Likewise, the coefficents on a particular monomial in $B_{n}(\mathbold{x}, \mathbold{y};1)$ are given by 
Viennot's theory by using the weights $a_k=2(k+1), b_k=4k+2,$ and $c_k = 2k$.  
It is a simple extension to show that a monomial 
$\mathbold{x}^{\mathcal{DT}} \mathbold{y}^{\mathcal{AT}}$ in $B_{n}(\mathbold{x}, \mathbold{y};q)$ has a 
$q$-coefficient given by weighting $NE$ steps by $a_k=(k+1)(1+q)$, $\overline{E}$ steps by 
$\overline{b}_k=1+k(1+q)$, $\underline{E}$ steps by $\underline{b}_k = q+k(1+q)$, and $SE$ steps by $c_k = k(1+q)$.  
It is also easy to extend these results to $G_{n,r}(\mathbold{x}, \mathbold{y};q)$.

\section*{Acknowledgements}
We thank Jim Haglund for helpful suggestions and Dennis Stanton for showing us the connection between the 
the number of monomials and moments of Laguerre polynomials (and its connections to Viennot's combinatorial theory).
We also thank the reviewers for useful suggestions that improved the presentation of the paper.
Finally, we thank the organizers of the $7$th Graduate Student Combinatorics Conference for
making it possible for the authors to meet and start this collaboration. 

\medskip

\noindent
\emph{Note added in proof.} C.~D.~Savage and the first author have resolved Conjecture~\ref{conj:Brenti} (and thus Conjecture~\ref{conj:brenti}) and 
part of Conjecture~\ref{conj:dps} in  \cite{SV13}.
\bibliographystyle{amsplain}
\bibliography{eulerpoly-v2} 
\appendix
\section{List of stable multivariate Eulerian polynomials}

\begin{align*}
A_0 (\mathbold{x},\mathbold{y}) &=1\\
A_1 (\mathbold{x},\mathbold{y}) &=x_2+y_2 \\
A_2 (\mathbold{x},\mathbold{y}) &= x_2 x_3+x_3 y_2+x_2 y_3+2 x_3 y_3+y_2 y_3\\
\\
B_1 (\mathbold{x},\mathbold{y}) &= x_1+y_1\\
B_2 (\mathbold{x},\mathbold{y}) &= x_1 x_2+x_2 y_1+x_1 y_2+4 x_2 y_2+y_1 y_2 \\
\\
B_1(\mathbold{x},\mathbold{y};q) &= q x_1+y_1 \\
B_2(\mathbold{x},\mathbold{y};q) &= q^2 x_1 x_2+q x_2 y_1+q x_1 y_2+(1+q)^2x_2 y_2 +y_1 y_2 \\
\\
\widetilde{A}_1(\mathbold{x},\mathbold{y}) &= 2 x_2 y_2\\
\widetilde{A}_2(\mathbold{x},\mathbold{y}) &= 2 x_2 x_3 y_3+2 x_3 y_2 y_3\\
\widetilde{A}_3(\mathbold{x},\mathbold{y}) &= 2 x_2 x_3 x_4 y_4+2 x_3 x_4 y_2 y_4+2 x_2 x_4 y_3 y_4+4 x_3 x_4 y_3 y_4+2 x_4 y_2 y_3 y_4\\
\\
\widetilde{C}_1(\mathbold{x},\mathbold{y}) &= 2 x_1 y_1\\
\widetilde{C}_2(\mathbold{x},\mathbold{y}) &= 4 x_1 x_2 y_2+4 x_2 y_1 y_2\\
\widetilde{C}_3(\mathbold{x},\mathbold{y}) &= 8 x_1 x_2 x_3 y_3+8 x_2 x_3 y_1 y_3+8 x_1 x_3 y_2 y_3+16 x_2 x_3 y_2 y_3+8 x_3 y_1 y_2 y_3\\
\end{align*}
\end{document}